\documentclass{amsart}
\RequirePackage[colorlinks,citecolor=blue,urlcolor=blue,linkcolor=blue]{hyperref}
\hypersetup{
	colorlinks = true,
	citecolor=blue,
	urlcolor=blue,
	linkcolor=blue,
	pdfpagemode = UseNone
}

\usepackage{graphicx,hyperref,xcolor}
\usepackage{enumerate}
\usepackage{geometry,soul,fancybox}
\usepackage{tikz}
\usetikzlibrary{patterns}
\newtheorem{theorem}{Theorem}[section]

\theoremstyle{definition}

\newtheorem{remark}[theorem]{Remark}

\numberwithin{equation}{section}

\usepackage{framed,pgf}
\newcounter{oldeq}
\newcounter{usesofarxiv}

\begin{document}
		\date{}
		\title[GIG random matrices and  a Yang-Baxter extension of the Matsumoto-Yor property]{GIG random matrices  \\ and \\  a Yang-Baxter extension of the Matsumoto-Yor property}
		\author{G. Letac, M. Piccioni, J. Wesołowski}%
		\address{G\'erard Letac, Institute de Math\'ematiques, Universit\'e de Toulouse, 31062 Toulouse, France;T\'eSA, 7 Bd de la Gare, 31500 Toulouse}
		\email{gerard.letac@math.univ-toulouse.fr}
		\address{Mauro Piccioni, Dipartimento di Matematica, Sapienza Universit\`a di Roma, Piazzale Aldo Moro 5, 00185 Roma, Italy.}
		\email{mauropiccioni129@gmail.com}
		\address{Jacek Weso{\l}owski, Wydzia{\l} Matematyki i Nauk Informacyjnych,	Politechnika Warszawska, ul. Koszykowa 75, 00-662 	Warszawa, Poland; G\l \'owny Urz\c ad Statystyczny, Al. Niepodleg{\l}o\'sci 208, 00-925 Warszawa}
		\email{jacek.wesolowski@pw.edu.pl}
		%\email{}%
		%
		%\thanks{}%
		%\subjclass{}%
		%\keywords{}%
		%
		%%\date{}%
		%%\dedicatory{}%
		%%\commby{}%
		%% ----------------------------------------------------------------
		
		% ----------------------------------------------------------------

		\begin{abstract}
					Sasada and Uozumi, \cite{SasUoz2024},  identified independence preserving $[2:2]$  quadrirational parametric Yang-Baxter maps, see \eqref{YBEQ}, on $(0,\infty)$. In particular, the map denoted there by $H_{III,B}^{(\alpha,\beta)}$, see \eqref{CS}, was connected to the independence preserving property of the GIG distributions on $(0,\infty)$. Remarkably, the property appears also naturally in probabilistic   integrable models of   discrete  Korteweg de Vries type, as observed by Croydon and Sasada,  \cite{CroSas2020}. In the case of $(\alpha,\beta)=(1,0)$ the independence  reduces to the classical Matsumoto-Yor property, \cite{MatYor2001}. In \cite{LetWes2024} we proposed an extension of $H_{III,B}^{(\alpha,\beta)}$ to a map on the cone of symmetric positive definite matrices of a fixed dimension, showing that such  extended map preserves independence of GIG random matrices. In the present paper we prove two results: (i)  the matrix GIG distributions are characterized by the independence property governed by this map;  (ii) the matrix variate extension of $H_{III,B}^{(\alpha,\beta)}$ we use,  is a parametric Yang-Baxter map.
		\end{abstract}
		
		\maketitle
		%\tableofcontents
			\section{{\bf INTRODUCTION}}
		Recall that a random variable $X$ has the generalized inverse Gaussian distribution, $\mathrm{GIG}(\lambda,\alpha,\beta)$, where $\lambda\in \mathbb R$ and $\alpha,\beta\in(0,\infty)$, if its density is of the form
		$$
		f(x)\propto x^{\lambda-1}\,e^{-\alpha x-\frac \beta x}\,\mathbf 1_{(0,\infty)}(x).
		$$
		Denote by $\Omega^{(r)}$ the Euclidean space of symmetric $r\times r$ matrices, with the inner product $\langle x,y\rangle=\mathrm{tr}(xy)$, $x,y\in\Omega^{(r)}$. By $\Omega^{(r)}_+$ denote the cone of positive definite elements of $\Omega^{(r)}$. 
		Since $r$ will be fixed in this paper we skip it in further notation, writing $\Omega$ and $\Omega_+$ instead. In this paper we will consider the matrix variate counterpart of $\mathrm {GIG}$ distribution, $\mathrm{MGIG}(\lambda,a,b)$, where $\lambda\in\mathbb R$, $a,b\in \Omega_+$, which is defined by the density
		\begin{equation}\label{mgig}
		f(x)\propto (\det x)^{\lambda-\frac{r+1}{2}}\,e^{-\langle a, x\rangle-\langle b, x^{-1}\rangle}\,\mathbf 1_{\Omega_+}(x).
		\end{equation}
		
		\bigskip
		The story of the problem we are interested in starts in the late 90's of the last century, when Matsumoto and Yor, while studying the conditional structure of the exponential Brownian motion, discovered the following independence preserving property for the GIG and gamma laws: If $(X,Y)\sim \mathrm{GIG}(-\lambda,\gamma_1,\gamma_2)\otimes\mathrm{Gamma}(\lambda,\gamma_1)$ and
		\begin{equation}\label{oMY}
		(U,V)=\left(\tfrac1{X+Y},\,\tfrac1X-\tfrac1{X+Y}\right)
		\end{equation}
		then $(U,V)\sim  \mathrm{GIG}(-\lambda,\gamma_2,\gamma_1)\otimes\mathrm{Gamma}(\lambda,\gamma_2)$. Here $\mathrm{Gamma}(\lambda,\gamma)$ distribution, $\lambda,\gamma>0$, is defined by the density $f(x)\propto x^{\lambda-1}e^{-\gamma x}\mathbf 1_{(0,\infty)}(x)$. The result appeared in \cite{MatYor2001}. 
		
		\bigskip
		A matrix variate version of the MY property was presented in \cite{LetWes2000}.  The role of the Gamma law was taken over by the Wishart distribution $\mathrm{W}(\lambda,c)$, $\lambda>\tfrac{r-1}{2}$, $c\in\Omega_+$, defined by the density $f(x)\propto (\det\,x)^{\lambda-\frac{r+1}{2}}\,e^{-\langle c,x\rangle}\,\mathbf 1_{\Omega_+}(x)$.  Actually, the  $\mathrm{W}(\lambda,c)$ is defined for $\lambda\in \Lambda=\left\{0,\tfrac12,\tfrac22,\ldots,\tfrac{r-1}{2}\right\}\cup\left(\tfrac{r-1}{2},\infty\right)$, but for $\lambda\le \tfrac{r-1}{2}$ it does not have a density and is defined on the closure $\bar\Omega_+$ of $\Omega_+$. The property reads: let $\lambda\in \Lambda$ and $a,b\in\Omega_+$. If $(X,Y)\sim\mathrm{MGIG}(-\lambda,a,b)\otimes \mathrm W(\lambda,a)$ and
		 \begin{equation}\label{mMY}
		 (U,V)=\left((X+Y)^{-1},\,X^{-1}-(X+Y)^{-1}\right)
		 \end{equation}
		then $(U,V)\sim\mathrm{MGIG}(-\lambda,b,a)\otimes \mathrm W(\lambda,b)$. In the same paper the authors proved a characterization of MGIG and Wishart distribution by the above independence preserving property, under the assumption that $X$ and $Y$ have densities which are of class $C_2$ and are strictly positive on $\Omega_+$. 
		
		Let us come back to the univariate case. An unexpected connection of the MY property with the discrete Korteweg de Vries (KdV) integrable model, was discovered  by Croydon and Sasada in \cite{CroSas2020}. These authors  observed that if $(X,Y)\sim\mathrm{GIG}(-\lambda,\alpha\gamma_1,\gamma_2)\otimes\mathrm{GIG}(-\lambda,\beta\gamma_2,\gamma_1)$ for $\lambda\in\mathbb R$, $\gamma_1,\gamma_2>0$, and
		\begin{equation}\label{CS}
		(U,V)=H_{III,B}^{(\alpha,\beta)}(X,Y)=\left(Y\tfrac{1+\beta XY}{1+\alpha XY},\,X\tfrac{1+\alpha XY}{1+\beta XY}\right),
		\end{equation}
		then $(U,V)\sim\mathrm{GIG}(-\lambda,\alpha\gamma_2,\gamma_1)\otimes\mathrm{GIG}(-\lambda,\beta \gamma_1,\gamma_2)$, relating this property in the nontrivial case of $\alpha\neq \beta$ to the modified discrete KdV  model on $\mathbb Z^2$. Here we used the symbol $H_{III,B}^{(\alpha,\beta)}$ following  Sasada and Uozumi,  \cite{SasUoz2024}, who identified this transformation  as belonging to a hierarchy of  parametric quadrirational Yang-Baxter map (YB) preserving independence property - see \eqref{YBEQ} for the definition of YB maps. 
		
				\bigskip
		It is elementary to check that if $W\sim \mathrm{GIG}(\kappa,p,q)$ then $\tfrac1W\sim \mathrm{GIG}(-\kappa,q,p)$. Consequently, writing $1/Y$ instead of $Y$ and $1/V$ instead of $V$ in \eqref{CS} one obtains another presentation of the same independence property, getting a form which is more convenient for the purposes of the present paper:
		
		If $(X,Y)\sim \mathrm{GIG}(-\lambda,\alpha\gamma_1,\gamma_2)\otimes\mathrm{GIG}(\lambda,\gamma_1,\beta\gamma_2)$ and
		\begin{align}\label{GIG1}
		(U,V)=&\left(\tfrac{\beta X+Y}{Y(\alpha X+Y)},\,\tfrac{\beta X+Y}{X(\alpha X+Y)}\right)\\
		=&\left(\tfrac \beta\alpha Y^{-1}+\left(1-\tfrac \beta\alpha\right)(\alpha X+Y)^{-1},\,X^{-1}+(\beta-\alpha)(\alpha X+Y)^{-1}\right)\nonumber
		\end{align}
		then $(U,V)\sim\mathrm{GIG}(-\lambda,\alpha\gamma_2,\gamma_1)\otimes\mathrm{GIG}(\lambda,\gamma_2,\beta \gamma_1)$. 
		%Note that \eqref{GIG1} does not define a YB map any more. 
		For further use we observe that this last  independence property  is equivalent to the fact that the quadruple of  densities $\mathfrak f_X$, $\mathfrak f_Y$, $\mathfrak f_U$ and $\mathfrak f_V$ (in case $X$, $Y$, and thus $U$ and $V$, have densities) satisfies the  functional equation 
		\begin{equation}\label{eff}
			 \mathfrak d_{U}\left(\tfrac{\tau+\beta \sigma}{\tau+\alpha \sigma}\tfrac{1}{\tau}\right)\, \mathfrak d_{V}\left(\tfrac{\tau+\beta \sigma}{\tau+\alpha \sigma}\tfrac{1}{\sigma}\right)\,\left(\tfrac{\tau+\beta\sigma}{\tau\sigma(\tau+\alpha\sigma)}\right)^2=\, \mathfrak d_{X}(\sigma)\ \mathfrak d_{Y}(\tau),\quad \sigma,\tau\in(0,\infty),\quad \tau,\sigma>0,
		\end{equation}
		  for $ \mathfrak d_{W}= \mathfrak f_{W}$, $W\in\{X,Y,U,V\}$. 
		  
		  \bigskip
		  Note also that for $\alpha=1$ and $\beta=0$ this independence property  for $(U,V)$ defined in \eqref{GIG1}  reduces to the original MY property \eqref{oMY}: see the second line of \eqref{GIG1}. In this context it is interesting to note that \eqref{CS} for $\beta=0\neq \alpha$ refers just to the straight discrete KdV model, see e.g. \cite{CroSasTsu2022}. 
		  
		  In \cite{CroSas2020} the authors conjectured that the above independence preserving property with $(U,V)$ defined by \eqref{CS}   characterizes GIG laws. Under assumptions that the distributions of $X$, $Y$ have strictly positive and twice differentiable densities the conjecture was proved in \cite{BaoNoa2021}. The complete solution with no assumptions regarding existence of densities was given in \cite{LetWes2024}. 
		
		\bigskip
		In \cite{LetWes2024}, extending the matrix variate version of the MY property from \cite{LetWes2000},  the modified KdV property  was established for random matrices as follows:
		
		\bigskip
		\begin{theorem}\cite{LetWes2024}\label{direc}
		Let $X$ and $Y$ be independent $\Omega_+$-valued random matrices with distributions $\mathrm{MGIG}(\lambda,\alpha a,b)$ and $\mathrm{MGIG}(\lambda,\beta b,a)$, respectively. Let $\phi^{(\alpha,\beta)}:\Omega_+^2\to\Omega_+^2$ be defined by 
		\begin{equation}\label{faj}
			\phi^{(\alpha,\beta)}(x,y)=\left(y(I+\alpha xy)^{-1}(I+\beta xy),\,x(I+\beta yx)^{-1}(I+\alpha yx)\right)
		\end{equation}
		and let
		$$
		\left( U,\, V\right)=\phi^{(\alpha,\beta)}(X,Y)
.		$$

\vspace{1mm}
		Then $U$ and $ V$  are independent, $ U\sim\mathrm{MGIG}(\lambda,\alpha b,a)$ and $ V=\mathrm{MGIG}(\lambda,\beta a,b)$.
		\end{theorem}	
		
		\bigskip
		The main goal of the present paper is to prove a converse result, that is to prove that MGIG laws are characterized by the independence property from Theorem \ref{direc}. Of course, the question is meaningful only in case $\alpha\neq \beta$ since  $\phi^{(\alpha,\alpha)}(x,y)=(y,x)$. To this end we will need technical asumptions of existence and smoothness properties of densities:
		
		\bigskip
		\begin{theorem}\label{chara} 
			Let $X$ and $Y$ be independent $\Omega_+$ valued random matrices with densities which are strictly positive  and twice differentiable on $\Omega_+$. Let $(U,V)=\phi^{(\alpha,\beta)}(X,Y)$ for some $\alpha\ne \beta$, where $\phi^{(\alpha,\beta)}$ is defined in \eqref{faj}. 
			
			Then there exist $\lambda\in\mathbb R$ and $a,b\in\Omega_+$ such that
			$$
			(X,Y)\sim \mathrm{MGIG}(\lambda,\alpha a,b)\otimes \mathrm{MGIG}(\lambda,\beta b,a)\quad\mbox{and}\quad 	(U,V)\sim \mathrm{MGIG}(\lambda,\alpha b,a)\otimes \mathrm{MGIG}(\lambda,\beta a,b)
			$$
			\end{theorem}	
			
			\bigskip
			The proof of Theorem \ref{chara} is given in Section 2. Actually,  for convenience we will state and prove there an equivalent result, Theorem \ref{charac}. A similar equivalent approach appeared to be  convenient  in the univariate case, see  \cite{LetWes2024}.
		
		\bigskip
		The map  $\phi^{(\alpha,\beta)}$ is a direct extension of  $H_{III,B}^{(\alpha,\beta)}$ from the domain $(0,\infty)^2$ to $\Omega_+^2$, that is from a commutative setting to a non-commutative one.   Since $H_{III,B}^{(\alpha,\beta)}$  is a parametric YB map, see Theorem 2 in \cite{PSTV2010},   it is natural  to ask if $\phi^{(\alpha,\beta)}$ defined in \eqref{faj} is a YB map on $\Omega_+^2$. It appears that it really is: 
		
		\bigskip
		\begin{theorem}\label{YBTH}
			The function $\phi^{(\cdot,\cdot)}=\left(\phi_1^{(\cdot,\cdot)},\,\phi_2^{(\cdot,\cdot)}\right)$ is a parametric Yang-Baxter map, that is
			\begin{equation}\label{YBEQ}
				F_{12}^{(\alpha,\beta)}\circ F_{13}^{(\alpha,\gamma)}\circ F_{23}^{(\beta,\gamma)}=F_{23}^{(\beta,\gamma)}\circ F_{13}^{(\alpha,\gamma)}\circ F_{12}^{(\alpha,\beta)},
			\end{equation}	
			with 
			$$
			\left.\begin{cases}
				F_{12}^{(\alpha,\beta)}(x,y,z)=\left(\phi_1^{(\alpha,\beta)}(x,y),\,\phi_2^{(\alpha,\beta)}(x,y),z\right)\\
				F_{13}^{(\alpha,\gamma)}(x,y,z)=\left(\phi_1^{(\alpha,\gamma)}(x,z),\,y,\,\phi_2^{(\alpha,\gamma)}(x,z)\right)\\
				F_{23}^{(\beta,\gamma)}(x,y,z)=\left(x,\,\phi_1^{(\beta,\gamma)}(y,z),\phi_2^{(\beta,\gamma)}(y,z)\right)
			\end{cases}
			\right\}
			\quad x,y,z\in\Omega_+.
			$$
		\end{theorem}	
		\bigskip
		The proof of Theorem \ref{YBTH} is given in the Appendix. 
		
	%In view of Theorem \ref{YBTH} the map  $\phi^{(\alpha,\beta)}$ can be seen as an extension of  $H_{III,B}^{(\alpha,\beta)}$, a function on $(0,\infty)^2$, to a function on $\Omega_+^2$, which preserves the Yang-Baxter property. 
		
%	Our main goal in this paper is to prove the converse of Theorem \ref{direc}, i.e. to characterize the matrix GIG distributions by the independence preserving property from this theorem. This is done in Theorem  \ref{charac}. To this end we will need the technical assumptions that the densities of $X$ and $Y$ are strictly positive on $\Omega_+$ and that they are twice differentiable.
		
	\bigskip	
	\section{{\bf CHARACTERIZATION AND ITS PROOF}}
	Similarly as in the univariate case, for the matrix variate MGIG variable, see \eqref{mgig}, the following property holds: if $W\sim\mathrm{MGIG}(\kappa,c,d)$ then $W^{-1}\sim \mathrm{MGIG}(-\kappa,d,c)$. Consequently, exactly following the transition  in the univariate case  from presentation  \eqref{CS} to presentation \eqref{GIG1} by taking $Y$ in place of $Y^{-1}$ and $V$ in place of $V^{-1} $,  we see that  Theorem \ref{chara} is equivalent to Theorem \ref{charac} below. It is clear that  independence properties of $X$ and $Y$ as well as of $U$ and $V$ of Theorem \ref{chara} are preserved under these transformations of $Y$ and $V$.
	
	\bigskip
	\begin{theorem}\label{charac}
		Assume that $X$ and $Y$ are independent random matrices valued in $\Omega_+$ having strictly positive twice differentiable densities on $\Omega_+$. Let $\alpha,\beta\ge 0$ be two distinct numbers. Let 
		\begin{equation}\label{UV0}
		\begin{cases} U=(\alpha X+Y)^{-1}(\beta X+Y)Y^{-1}
		\\
			V=(\alpha X+Y)^{-1}(\beta X+Y)X^{-1}.
			\end{cases}
		\end{equation}
		
		\bigskip
		If $U$ and $V$ are independent then there exist $\lambda\in\mathbb R$ and $a,b\in\Omega_+$ such that
			\begin{equation}\label{XY}
			X\sim\mathrm{MGIG}(\lambda,\alpha a,b)\quad\mbox{and}\quad Y\sim\mathrm{MGIG}(-\lambda,a,\beta b)
		\end{equation}
		and
		\begin{equation}\label{UV}
		U\sim\mathrm{MGIG}(\lambda,\alpha b,a)\quad\mbox{and}\quad V\sim\mathrm{MGIG}(-\lambda,b,\beta a).
		\end{equation}
	\end{theorem}	

\bigskip
\begin{remark}
	Note that $U$ and $V$, as defined in \eqref{UV0}, can be represented as 
		\begin{equation}\label{UV1}
		\begin{cases} U=\tfrac \beta\alpha Y^{-1}+\left(1-\tfrac \beta\alpha\right)(\alpha X+Y)^{-1}
			\\
			V=X^{-1}+(\beta-\alpha)(\alpha X+Y)^{-1},
		\end{cases}
	\end{equation}
	which actually is of the same form as  in the univariate case given in the  second line of \eqref{GIG1} . 
	
		Let us point out the following: If $X$ and $Y$ are symmetric positive definite matrices (spd) then $U$ and $V$  are also spd. Symmetry is clear from the  presentation of $U$ and $V$ in \eqref{UV1}. Positive definitness is less obvious. Several proofs are available. For instance one can refer to a theorem by Wigner, \cite{Wig1963}, which says that if $a,b,c$ are spd matrices of the same dimension and the product $abc$ is symmetric, then $abc$ is spd as well. Applying Wigner's result to the presentation of $U$ and $V$  in \eqref{UV0} shows that they are both  $\Omega_+$-valued random matrices.
\end{remark}
		
\bigskip
\begin{proof}[Proof of Theorem \ref{charac}.] It is divided into ten parts:
	
	\subsection{A functional equation on $\Omega_+$ and its counterpart on $(0,\infty)$.}
	
	$\,$
	
	\bigskip
	Note that in view of \eqref{UV1} we have $(U,V)=\psi(X,Y)$, where $\psi:\Omega_+^2\to\Omega_+^2$ is defined by
	$$
	\psi(x,y)=\left((\alpha x+y)^{-1}(\beta x+y)y^{-1},\,(\alpha x+y)^{-1}(\beta x+y)x^{-1}\right).
	$$
	For computing the jacobian of $\psi$ we note that $\psi=\theta\circ\phi^{(\alpha,\beta)}\circ\theta$, where $\Omega_+\ni(x,y)\stackrel{\theta}{\mapsto}(x,y^{-1})$. Recall, see Lemma 5.1, \cite{LetWes2024}, that the function  $\phi^{(\alpha,\beta)}=:\phi=(\phi_1,\phi_2)$ is an involution with jacobian $J_{\phi}=1$.  The jacobian of $\theta(x,y)$ is $(\det\,y)^{-(r+1)}$. Indeed, it follows e.g. from Lemma 2.2 in \cite{LetWes2000}, in view of the fact that the differential $D_y\,y^{-1}[h]=-y^{-1}hy^{-1}$, $h\in\Omega$. Since $\theta$ is an involution, we conclude that $\psi $ is an involution as well. Thus the jacobian of $\psi$ assumes the form
	$$
	J_\psi(x,y)=J_{\theta}(\phi(x,y^{-1}))\,J_\phi(x,y^{-1})\,J_{\theta}(x,y)=(\det(\phi_2(x,y^{-1})))^{-(r+1)}\,(\det\,y)^{-(r+1)}=\left(\tfrac{\det\,v}{\det\,y}\right)^{r+1}.
	$$
	
	\bigskip
	For $W\in\{X,Y,U,V\}$ we denote by $\mathfrak f_W$ the density of the law of $W$. Then let $\mathfrak g_W(x)=\log\,\mathfrak f_W(x)$ for $W\in\{X,U\}$ and $\mathfrak g_W(x)=\log\,\mathfrak f_W(x)+(r+1)\log \det x$ for $W\in\{Y,V\}$. Consequently, the independence assumption implies
	\begin{equation}\label{gs}
	\mathfrak g_U(u)+\mathfrak g_V(v)=\mathfrak g_X(x)+\mathfrak g_Y(y),\quad x,y\in\Omega_+,
	\end{equation}
	where $(u,v)=(u(x,y),v(x,y))$.
	
	\bigskip
	We choose an arbitrary  $w\in\Omega_+$ and apply \eqref{gs} to $(x,y)=(\sigma w,\,\tau w)$, where $\sigma,\tau\in(0,\infty)$. It gives 
	\begin{equation}\label{HYPST}
		\mathfrak g_U\left((\tau w+\alpha \sigma w)^{-1}(\tau w+\beta \sigma w)(\tau w)^{-1}\right)+\mathfrak g_V\left((\tau w+\alpha \sigma w)^{-1}(\tau w+\beta \sigma w)(\sigma w)^{-1}\right)=\mathfrak g_X(\sigma w)+\mathfrak g_Y(\tau w).
	\end{equation}
	
	\bigskip
	In the sequel we will consider several quadruples of real functions $(\mathfrak m_X,\,\mathfrak m_Y, \mathfrak m_U, \mathfrak m_V)=(\mathfrak m_W)_{W\in\{X,Y,U,V\}}$ defined on $(0,\infty)$, satisfying 
	\begin{equation}\label{HYPSTR}
		\mathfrak m_U\left(\tfrac{\tau+\beta \sigma}{\tau+\alpha \sigma}\tfrac{1}{\tau}\right)+\mathfrak m_V\left(\tfrac{\tau+\beta \sigma}{\tau+\alpha \sigma}\tfrac{1}{\sigma}\right)=\mathfrak m_X(\sigma)+\mathfrak m_Y(\tau),\quad \sigma,\tau>0,
	\end{equation}
	
	For every $w\in\Omega_+$ we define the quadruple of real functions 
	$$\begin{cases}\mathfrak h_U^{(w)}(\rho):= &
		\mathfrak g_U(\rho w^{-1}), \quad \rho>0, \\
		\mathfrak h_V^{(w)}(\rho):=&\mathfrak g_V(\rho w^{-1}), \quad \rho>0,\\
		\mathfrak h_X^{(w)}(\rho):=&\mathfrak g_X(\rho w), \quad \rho>0,\\
		\mathfrak h_Y^{(w)}(\rho):=&\mathfrak g_Y(\rho w),\quad  \rho>0.
	\end{cases}
$$
	 In view of  \eqref{HYPST} for all $w\in\Omega_+$ the  quadruple $(\mathfrak h^{(w)}_W)_{W\in\{X,Y,U,V\}}$  satisfies \eqref{HYPSTR} .

\bigskip
\subsection{Solution of equation \eqref{HYPSTR} for the quadruple $(\mathfrak h^{(w)}_W)_{W\in\{X,Y,U,V\}}$.} $\,$

 \bigskip
We attack the problem by reducing equation \eqref{HYPSTR} in this case to equation \eqref{eff} for suitably chosen univariate densities.

\bigskip
Note that for any $W\in\{X,Y,U,V\}$ and any $p\in\mathbb R$, $q,s>0$ we have
\begin{equation}\label{finite}
\infty>\int_0^\infty\,\left(\int_{\Omega_+}\,\mathfrak f_W(w)\,\mathrm dw\right)\,\xi^{p-1}e^{-q\xi-s/\xi}\,\mathrm d\xi=\int_{\Omega_+}\,\left(\int_0^\infty\,\xi^{p+r-1}\,e^{-q\xi-s/\xi}\mathfrak f_W(\xi v)\,\mathrm d\xi\right)\,\mathrm dv.
\end{equation}
Therefore, the function $(0,\infty)\ni\xi\mapsto \xi^{p+r-1}e^{-q\xi-s/\xi}\mathfrak f_W(\xi v)\in(0,\infty)$ is integrable for almost all $v\in\Omega_+$  with respect to the Lebesgue measure.

\bigskip
Consider independent $X'$ and $Y'$ univariate $\mathrm{GIG}(-\lambda',\alpha\gamma_1,\gamma_2)$ and  $\mathrm{GIG}(\lambda',\gamma_1,\beta\gamma_2)$. Then $(U',V')$ defined in \eqref{GIG1} are independent $\mathrm{GIG}(-\lambda',\alpha\gamma_2,\gamma_1)$ and $\mathrm{GIG}(\lambda',\gamma_2,\beta\gamma_1)$. Let $\tilde {\mathfrak h}_{W}=\log\,\mathfrak f_{W}$ for $W\in\{X',U'\}$ and $\tilde{\mathfrak h}_{W}(\tau)=\log\,\mathfrak f_{W}(\tau)+2\log\,\tau$ for $W\in\{Y',V'\}$, where $\mathfrak f_{W}$ is the density of $W\in\{X',Y',U',V'\}$. Consequently, the quadruple $(\tilde{\mathfrak h}_{W})_{W\in\{X',Y',U',V'\}}$ satisfies \eqref{HYPSTR}. 

\bigskip
 Set 
 $$
 \mathfrak H_X^{(w)}:=\tilde{\mathfrak h}_{X'}+\mathfrak h_X^{(w)},\quad \mathfrak H_Y^{(w)}:=\tilde{\mathfrak h}_{Y'}+\mathfrak h_Y^{(w)},\quad\mathfrak H_U^{(w)}:=\tilde{\mathfrak h}_{U'}+\mathfrak h_U^{(w)},\quad\mathfrak H_V^{(w)}:=\tilde{\mathfrak h}_{V'}+\mathfrak h_V^{(w)},\quad w\in\Omega_+,
 $$ 
 and define  functions $\mathfrak q_W^{(w)}$, $W\in\{X,Y,U,V\}$, $w\in\Omega_+$, as follows
$$
\mathfrak q_X^{(w)}(\rho):=\exp(\mathfrak H_X^{(w)}(\rho))=\mathfrak f_{X'}(\rho)\,\mathfrak f_X(\rho w),\quad
\mathfrak q_U^{(w)}(\rho):=\exp(\mathfrak H_U^{(w)}(\rho))=\mathfrak f_{U'}(\rho)\,\mathfrak f_U(\rho w^{-1}),
$$
and 
\begin{align*}
\mathfrak q_Y^{(w)}(\rho)&:=\exp(\mathfrak H_Y^{(w)}(\rho))\rho^{-2}=\mathfrak f_{Y'}(\rho)\,\mathfrak f_Y(\rho w)\,\rho^{r(r+1)}(\det\,w)^{r+1},\\
\mathfrak q_V^{(w)}(\rho)&:=\exp(\mathfrak H_V^{(w)}(\rho))\rho^{-2}=\mathfrak f_{V'}(\rho)\,\mathfrak f_V(\rho w^{-1})\,\rho^{r(r+1)}(\det\,w)^{-r-1},
\end{align*}
$\rho>0$. Since for $w\in\Omega_+$ the quadruple $(\mathfrak H_W^{(w)})_{W\in\{X,Y,U,V\}}$ clearly satisfies \eqref{HYPSTR} we see that $\mathfrak q_X^{(w)}$, $\mathfrak q_Y^{(w)}$, $\mathfrak q_U^{(w)}$ and $\mathfrak q_V^{(w)}$,   satisfy \eqref{eff} with $(\mathfrak d_W)_{W\in\{X,Y,U,V\}}=(\mathfrak q_W^{(w)})_{W\in\{X,Y,U,V\}}$, namely 

\begin{equation}\label{efs}
	\mathfrak q_{U}^{(w)}\left(\tfrac{\tau+\beta \sigma}{\tau+\alpha \sigma}\tfrac{1}{\tau}\right)\,\mathfrak q_{V}^{(w)}\left(\tfrac{\tau+\beta \sigma}{\tau+\alpha \sigma}\tfrac{1}{\sigma}\right)\,\left(\tfrac{\tau+\beta\sigma}{\tau\sigma(\tau+\alpha\sigma)}\right)^2=\,\mathfrak q_{X}^{(w)}(\sigma)\,\mathfrak q_{Y}^{(w)}(\tau),\quad \sigma,\tau\in(0,\infty),
\end{equation}   

\bigskip\noindent
In view of \eqref{finite}, (nonnegative) functions $(\mathfrak q_W^{(w)})_{W\in\{X,Y,U,V\}}$,  are integrable over $(0,\infty)$ for almost all  $w\in\Omega_+$ (with respect to the Lebesgue measure on $\Omega_+$). For such $w\in\Omega_+$ set $\mathfrak c_W^{(w)}:=\int_0^\infty\,\mathfrak q_W^{(w)}(\xi)\,\mathrm d\xi$, $W\in\{X,Y,U,V\}$. Integrating \eqref{efs} with respect to $(\tau,\sigma)\in(0,\infty)^2$ we get for almost all $w\in\Omega_+$
\begin{equation}\label{cis}
\mathfrak c_X^{(w)}\,\mathfrak c_Y^{(w)}=\mathfrak c_U^{(w)}\,\mathfrak c_V^{(w)}.
\end{equation}

\bigskip
 Let $W''$ be a univariate random variable with density $\mathfrak f_{W''}^{(w)}:=\mathfrak q_W^{(w)}/\mathfrak c_W^{(w)}$, $W\in\{X,Y,U,V\}$ for almost all $w\in\Omega_+$.  Then \eqref{efs} and \eqref{cis} yield  \eqref{eff} for $(\mathfrak d_W)_{W\in \{X,Y,U,V\}}=(\mathfrak f^{(w)}_{W''})_{W\in \{X,Y,U,V\}}$. 
 
\bigskip
  Thus, if $X''$ and $Y''$ are independent and  $(U'',V'')$ are defined by \eqref{GIG1} with $(X,Y)$ changed into $(X'',Y'')$, then equality \eqref{eff} for $(\mathfrak d_W)_{W\in \{X,Y,U,V\}}=(\mathfrak f^{(w)}_{W''})_{W\in \{X,Y,U,V\}}$ implies that $U''$ and $V''$ are also independent.  In view of the univariate characterization of GIG from \cite{LetWes2024} we conclude that $X''$, $Y''$, $U''$ and $V''$ have respective GIG distributions: in particular, there exist seven functions  $\mathfrak l'':\Omega_+ \to \mathbb R$, $\mathfrak a'', \mathfrak b'':\Omega_+\to (0,\infty)$ and $\mathfrak C_W'':\Omega_+\to\mathbb R$, $W\in\{X,Y,U,V\}$, such that for $\tau>0$ and   $w\in\Omega_+$ (almost all with respect to the Lebesgue measure)
	\begin{eqnarray*}
	\mathfrak H_U^{(w)}(\tau)&=&-\mathfrak l''(w) \log \tau-\alpha \mathfrak b''(w)\tau-\tfrac{\mathfrak a''(w)}{\tau}+ \mathfrak C_U''(w),\\
	\mathfrak H_V^{(w)}(\tau)&=&\mathfrak l''(w)  \log \tau- \mathfrak b''(w)\tau-\beta\tfrac{\mathfrak a''(w)}{\tau}+ \mathfrak C_V''(w),\\
	\mathfrak H_X^{(w)}(\tau)&=&-\mathfrak l''(w)  \log \tau-\alpha\mathfrak  a''(w)\tau-\tfrac{\mathfrak b''(w)}{\tau}+ \mathfrak C_X''(w),\\
	\mathfrak H_Y^{(w)}(\tau)&=&\mathfrak l''(w) \log \tau- \mathfrak a''(w)\tau-\beta \tfrac{\mathfrak b''(w)}{\tau}+ \mathfrak C_Y''(w).
\end{eqnarray*}

\bigskip
	Since $\mathfrak h_W^{(w)}=\mathfrak H_W^{(w)}-\tilde{\mathfrak h}_{W'}$ for $W\in\{X,Y,U,V\}$ we conclude that there exist seven functions, namely  $\mathfrak l=:\Omega_+ \to \mathbb R$, $\mathfrak a, \mathfrak b:\Omega_+\to (0,\infty)$ and $\mathfrak C_W:\Omega_+\to\mathbb R$, $W\in\{X,Y,U,V\}$, such that for $\tau>0$ and   $w\in\Omega_+$ (almost all with respect to the Lebesgue measure on $\Omega_+$) 
	\begin{eqnarray}
		\mathfrak h_U^{(w)}(\tau)&=&-\mathfrak l(w) \log \tau-\alpha \mathfrak b(w)\tau-\tfrac{\mathfrak a(w)}{\tau}+ \mathfrak C_U(w),\label{hu}\\
		\mathfrak h_V^{(w)}(\tau)&=&\mathfrak l(w)  \log \tau- \mathfrak b(w)\tau-\beta\tfrac{\mathfrak a(w)}{\tau}+ \mathfrak C_V(w)\label{hv}\\
		\mathfrak h_X^{(w)}(\tau)&=&-\mathfrak l(w)  \log \tau-\alpha\mathfrak  a(w)\tau-\tfrac{\mathfrak b(w)}{\tau}+ \mathfrak C_X(w)\label{hx}\\
		\mathfrak h_Y^{(w)}(\tau)&=&\mathfrak l(w) \log \tau- \mathfrak a(w)\tau-\beta \tfrac{\mathfrak b(w)}{\tau}+ \mathfrak C_Y(w).\label{hy}
	\end{eqnarray}
Since $\mathfrak h_X^{(w)}(\tau)$, $\mathfrak h_Y^{(w)}(\tau)$, $\mathfrak h_U^{(w)}(\tau)$ and $\mathfrak h_V^{(w)}(\tau)$ for $\tau>0$ are continuous functions of $w\in\Omega_+$ it follows that the equalities \eqref{hu} - \eqref{hy} are satisfied for all $w\in\Omega_+$

\bigskip
\subsection{Scaling arguments of $\mathfrak a$, $\mathfrak b$, $\mathfrak l$ and $\mathfrak C_W$, $W\in\{X,Y,U,V\}$.}$\,$

\bigskip
Here and in subsection 2.4 we follow the idea from \cite{GerMisWes2013},  adapted in a clever way to the matrix version of the Matsumoto-Yor property in \cite{Kol2017}. The idea  lies in changing original $w$ to $\rho w$ and then, essentially, of using linear independence of the functions of scalar argument, which are involved in the equation, see e.g.  the derivations of formuals (10)-(13) in  \cite{Kol2017}.

\bigskip
For $W\in \{U,V\}$, $w\in\Omega_+$ and $\rho>0$ we have 
$$
\mathfrak h_W^{(\rho w)}(\rho\tau)=\mathfrak g_W(\tau w^{-1})=\mathfrak h_W^{(w)}(\tau), \quad\tau>0.
$$
Combining these identities with \eqref{hu} and \eqref{hv}, for $\rho,\tau>0$ we get 
	\begin{equation*}
		-\mathfrak l(w) \log \tau-\alpha \mathfrak b(w)\tau-\tfrac{\mathfrak a(w)}{\tau}+ \mathfrak C_U(w)=-\mathfrak l(\rho w) \log \rho \tau-\alpha \mathfrak b(\rho w)\rho\tau-\tfrac{\mathfrak a(\rho w)}{\rho\tau}+ \mathfrak C_U(\rho w)\nonumber
	\end{equation*}
	and
	\begin{equation*}
		\mathfrak l(w) \log \tau-\mathfrak b(w)\tau-\beta\tfrac{\mathfrak a(w)}{\tau}+ \mathfrak C_V(w)=\mathfrak l(\rho w) \log \rho \tau- \mathfrak b(\rho w)\rho \tau-\beta\tfrac{\mathfrak a(\rho w)}{\rho\tau}+ \mathfrak C_V(\rho w).
	\end{equation*}

\bigskip
	Similarly, for $W\in\{X,Y\}$, $w\in\Omega_+$ and $\rho>0$ we have  
	$$
	\mathfrak h_W^{(\rho w)}(\tau/\rho)=\mathfrak g_W( \tau w)=\mathfrak h_W^{(w)}(\tau), \quad\tau>0.
	$$
	Combining these identities with \eqref{hx} and \eqref{hy}, for $\rho,\tau>0$ we get
	\begin{equation*}
		-\mathfrak l(w) \log \tau-\alpha \mathfrak a(w)\tau-\tfrac{\mathfrak b(w)}{\tau}+ \mathfrak C_X(w)=-\mathfrak l(\rho w) \log \tfrac\tau\rho-\alpha \tfrac{\mathfrak a(\rho w)}{\rho}\,\tau-\tfrac{\rho \mathfrak b(\rho w)}{\tau}+ \mathfrak C_X(\rho w)
	\end{equation*} 
	and
	\begin{equation*}
		\mathfrak l(w) \log \tau- \mathfrak a(w)\tau-\beta \tfrac{\mathfrak b(w)}{\tau}+ \mathfrak C_Y(w)=l\mathfrak (\rho w) \log \tfrac\tau\rho- \frac{\mathfrak a(\rho w)}{\rho}\,\tau-\beta \tfrac{\rho \mathfrak b(\rho w)}{\tau}+ \mathfrak C_Y(\rho w).
	\end{equation*} 

\bigskip
Since $\tau\mapsto 1,\tau,\tau^{-1},\log \tau$ are linearly independent the above equalities imply that for all $w\in\Omega_+$ and all $\rho>0$ 
	\begin{eqnarray}\label{LIN1}
		\mathfrak a(\rho w)&=&\rho\,\mathfrak a(w)\\\label{LIN2}
		\mathfrak b(\rho w)&=&\tfrac1\rho\,\mathfrak b(w)\\\label{LIN3}
		\mathfrak l(\rho w)&=&\mathfrak l(w)\\\label{LIN4}
		\mathfrak C_U(\rho w)&=&\mathfrak C_U(w)+\mathfrak l(w)\log \rho\\\label{LIN5}
		\mathfrak C_V(\rho w)&=&\mathfrak C_V(w)-\mathfrak l(w)\log \rho\\\label{LIN6}
		\mathfrak C_X(\rho w)&=&\mathfrak C_X(w)-\mathfrak l(w)\log \rho\\\label{LIN7}
		\mathfrak C_Y(\rho w)&=&\mathfrak C_Y(w)+\mathfrak l(w)\log \rho\end{eqnarray}

\bigskip
\subsection{Separate equations for $\mathfrak a$, $\mathfrak b$, $\mathfrak l$ and $\mathfrak C_W$, $W\in\{X,Y,U,V\}$.}$\,$

\bigskip
	Taking  $\tau=1$  in \eqref{hu} - \eqref{hy} and referring to relations between $\mathfrak h_W$ and $\mathfrak g_W$ we get  
	
	\begin{eqnarray}
		\mathfrak g_U(x)&=&-\alpha  \mathfrak b(x^{-1})-\mathfrak a(x^{-1})+\mathfrak C_U(x^{-1})\label{gu}\\
		\mathfrak g_V(x)&=&- \mathfrak b(x^{-1})-\beta \mathfrak a(x^{-1})+\mathfrak C_V(x^{-1})\label{gv}\\
		\mathfrak g_X(x)&=&-\alpha  \mathfrak a(x)-\mathfrak b(x)+\mathfrak C_X(x)\label{gx}\\
		\mathfrak g_Y(x)&=&-  \mathfrak a(x)-\beta \mathfrak b(x)+\mathfrak C_Y(x).\label{gy}
	\end{eqnarray}  

\bigskip
We insert the above form of $\mathfrak g_U$, $\mathfrak g_V$, $\mathfrak g_X$ and $\mathfrak g_Y$ in \eqref{gs} replacing $x$ and $y$ by $\rho x$ and $\rho y$, which implies that $u=u(x,y)$ and $v=v(x,y)$ will be replaced by $u/\rho$ and $v/\rho$. Whence

\begin{eqnarray}&&\nonumber
	-\alpha \mathfrak b(\rho u^{-1})-\mathfrak a(\rho u^{-1})+\mathfrak C_U(\rho u^{-1}) - \mathfrak b(\rho v^{-1})-\beta \mathfrak a(\rho v^{-1})+\mathfrak C_V(\rho v^{-1})\\&=&
	-\alpha \mathfrak a(\rho x))-\mathfrak b(\rho x)+\mathfrak C_X(\rho x)\label{HYPINDEP3}
	-\mathfrak a(\rho y) - \beta \mathfrak b(\rho y)+\mathfrak C_Y(\rho y)\end{eqnarray} 
which in view of \eqref{LIN1}-\eqref{LIN7} can be rewritten as

\begin{align*}
	&- \rho\,\mathfrak a(u^{-1})-\tfrac{\alpha}\rho\, \mathfrak b(u^{-1})+\mathfrak C_U(u^{-1})+\mathfrak l(u^{-1})\log \rho
	- \beta\,\rho\,\mathfrak a(v^{-1})-\tfrac1\rho\, \mathfrak b(v^{-1})+\mathfrak C_V(v^{-1})-\mathfrak l(v^{-1})\log \rho\\
=&	- \tfrac1\rho\,\mathfrak b(x)-\alpha\,\rho\, \mathfrak a(x)+\mathfrak C_X(x)-\mathfrak l(x)\log \rho -\beta \tfrac1\rho\,\mathfrak b(y) -   \rho\,\mathfrak a(y)+\mathfrak C_Y(y)+\mathfrak l(y)\log \rho
\end{align*} 

\bigskip
Here again, we use  linear independence of the four functions $\rho\mapsto 1,\rho, 1/\rho,\log \rho$ to conclude that for all $x,y\in\Omega_+$ and $(u,v)=(u(x,y),\,v(x,y))$ 
\begin{eqnarray}
	\mathfrak a(u^{-1})+\beta \mathfrak a(v^{-1})&=&\alpha \mathfrak a(x)+ \mathfrak a(y),\label{as}\\
	\alpha \mathfrak b(u^{-1})+\mathfrak b(v^{-1})&=& \mathfrak b(x)+\beta \mathfrak b(y),\label{bs}\\
	\mathfrak C_U(u^{-1})+\mathfrak C_V(v^{-1})&=&\mathfrak C_X(x)+\mathfrak C_Y(y),\label{cs}\\
	\mathfrak l(u^{-1})-\mathfrak l(v^{-1})&=&\mathfrak l(y)-\mathfrak l(x).\label{lambda}
\end{eqnarray}

\bigskip
\subsection{Some observations about  derivatives $D_x$ and $D_y$.} $\,$

\bigskip
The goal is to identify functions $\mathfrak a$, $\mathfrak b$, $\mathfrak l$ and $\mathfrak C_W$, $W\in\{X,Y,U,V\}$, on the basis of equations \eqref{as}-\eqref{lambda}. To this end we will need differentials of $u^{-1}$ and $v^{-1}$ with respect to $x$ and $y$. To ease the notation below we write $z:=(\beta x+y)^{-1}$ remembering that it depends on $x$ and $y$.

\bigskip
In view of 
\begin{equation}\label{r1}
u^{-1}=y(\beta x+y)^{-1}(\alpha x+y)=yz\left(\tfrac\alpha\beta z^{-1}-\tfrac{\alpha-\beta}\beta y\right)=\tfrac\alpha\beta y-\tfrac{\alpha-\beta}{\beta}\,yzy
\end{equation} 
we have
\begin{equation}\label{dux}
	D_x\left(u^{-1}\right)(h)=(\alpha-\beta)yzhzy,\quad h\in\Omega.
\end{equation}
In view  of
\begin{equation}\label{r2}
v^{-1}=x(\beta x+y)^{-1}(\alpha x+y)=xz\left(z^{-1}+(\alpha-\beta)x\right)= x+(\alpha-\beta)\,xzx.
\end{equation}
we have
\begin{equation}\label{dvy}
	D_y\left(v^{-1}\right)(h)=-(\alpha-\beta)xzhzx,\quad h\in\Omega.
\end{equation}

\bigskip
Since
$$
\beta v^{-1}+u^{-1}=\alpha x+y,
$$
 \eqref{dux} and \eqref{dvy} imply
\begin{equation}\label{duy}
	D_y\left(u^{-1}\right)(h)=h+\beta(\alpha-\beta)xzhzx,\quad h\in\Omega
\end{equation}
and
\begin{equation}\label{dvx}
	D_x\left(v^{-1}\right)(h)=\tfrac\alpha\beta h-\tfrac{\alpha-\beta}{\beta}\,yzhzy,\quad h\in\Omega.
\end{equation}

\bigskip
Similarly, $zy=I-\beta z x$ leads to
\begin{equation}\label{Dy}
	D_y\left(zy\right)(h)=\beta zhzx\quad\mbox{and}\quad 
	D_y\left(yz\right)(h)=\beta xzhz.
\end{equation} 

\bigskip
\subsection{Identification of $\mathfrak a$.} $\,$

\bigskip
Applying $D_x$ to both sides of \eqref{as} and referring to \eqref{dux} and \eqref{dvx} we get
$$
(\alpha-\beta)zy\mathfrak a'\left(u^{-1}\right)yz+\alpha \mathfrak a'\left(v^{-1}\right)-(\alpha-\beta)zy\mathfrak a'\left(v^{-1}\right)yz=\alpha \mathfrak a'(x)
$$
which holds for all $x,y\in\Omega_+$. 

\bigskip
Applying $D_y$ to the above equation, referring to \eqref{duy}, \eqref{dvy} and \eqref{Dy}, we obtain
\begin{align*}
&\beta zhzx\mathfrak a'\left(u^{-1}\right)yz+\beta zy\mathfrak a'\left(u^{-1}\right)xzhz
+zy\left(\mathfrak a''\left(u^{-1}\right)\left[h+\beta(\alpha-\beta)xzhzx\right]\right)yz\\
-&\beta zhzx\mathfrak a'\left(v^{-1}\right)yz-\beta zy\mathfrak a'\left(v^{-1}\right)xzhz
+(\alpha-\beta)zy\left(\mathfrak a''\left(v^{-1}\right)\left[xzhzx\right]\right)yz-\alpha\mathfrak a''\left(v^{-1}\right)\left[xzhzx\right]=0,
\end{align*}
which can be rewritten as
\begin{align}
&(\alpha-\beta)zy\left(\beta \mathfrak a''\left(u^{-1}\right)\left[xzhzx\right]+\mathfrak a''\left(v^{-1}\right)\left[xzhzx\right]\right)yz
+zy\left(\mathfrak a''\left(u^{-1}\right)[h]\right)yz-\alpha\mathfrak a''\left(v^{-1}\right)\left[xzhzx\right]\label{eq:a}\\
=&\beta\left\{zhzx\left(\mathfrak a'\left(v^{-1}\right)-\mathfrak a'\left(u^{-1}\right)\right)yz+zy\left(\mathfrak a'\left(v^{-1}\right)-\mathfrak a'\left(u^{-1}\right)\right)xzhz\right\}.\nonumber
\end{align}

\bigskip
In view of \eqref{LIN1} we have $\mathfrak a'(\rho w)=\mathfrak a'(w)$ and $\mathfrak a''(\rho w)=\tfrac1\rho \mathfrak a''(w)$, $w\in\Omega_+$, $\rho>0$. 
Take $y=\rho x$. Then   
$$
u^{-1}=\rho v^{-1}=\tfrac{\rho(\alpha+\rho)}{\beta+\rho}\,x,
$$
which implies
$$
\mathfrak a'(v^{-1})=\mathfrak a'(u^{-1})\quad\mbox{and}\quad \mathfrak a''\left(u^{-1}\right)=\tfrac1\rho\,\mathfrak a''\left(v^{-1}\right)=\tfrac{\beta+\rho}{\rho(\alpha+\rho)}\,\mathfrak a''(x).
$$ 
Moreover,
$$
zy=yz=\tfrac \rho{\beta+\rho}\,I,\quad  zx=xz=\tfrac 1{\beta+\rho}\,I.
$$

\bigskip
So, for $y=\rho x$   equation \eqref{eq:a}  assumes the form
$$
 C(\rho)\tfrac{\beta+\rho}{\alpha+\rho}\,\mathfrak a''(x)[h]=0
$$
where
$$
C(\rho)=(\alpha-\beta)\tfrac{\rho^2}{(\beta+\rho)^2}\,\left(\tfrac\beta\rho \tfrac1{(\beta+\rho)^2}+\tfrac1{(\beta+\rho)^2}\right)+\tfrac{\rho^2}{(\beta+\rho)^2}\tfrac1\rho-\tfrac\alpha{(\beta+\rho)^2}=\tfrac{\rho^2-\alpha\beta}{(\beta+\rho)^3}.
$$

\bigskip
Taking $\rho^2\neq \alpha\beta$ we obtain $\mathfrak a''(x)=0$. Since $x\in\Omega_+$ is arbitrary, we conclude that there exist a matrix $a\in\Omega$ and number $\kappa$ such that $\mathfrak a(w)=\mathrm{tr}[aw]+\kappa$, $w\in\Omega_+$. In view of \eqref{LIN1} we see that $\kappa=0$.

\bigskip
\subsection{Identification of $\mathfrak b$.}$\,$

\bigskip
Define $\tilde{\mathfrak b}(w)=\mathfrak b(w^{-1})$. Then \eqref{LIN2} implies $\tilde{\mathfrak b}(\rho w)=\rho\tilde{\mathfrak b}(w)$, $w\in\Omega_+$, $\rho>0$. Moreover, \eqref{bs} yields 
$$
\alpha \tilde{\mathfrak b}(u)+\tilde{\mathfrak b}(v)=\tilde{\mathfrak b}(x^{-1})+\beta\tilde{\mathfrak b}(y^{-1}), \quad u,v\in\Omega_+
$$
with $x=x(u,v)=(\alpha u+v)^{-1}(\beta u+v)v^{-1}$ and $y=y(u,v)=(\alpha u+v)^{-1}(\beta u+v)u^{-1}$. Consequently, $\tilde{\mathfrak b}$ has the same property \eqref{LIN1} as $\mathfrak a$ and satisfies the same equation \eqref{as} as $\mathfrak a$. So, from the result on $a$ it follows that $\tilde{ \mathfrak b}(w)=\mathrm{tr}\left[bw\right]$, $w\in\Omega_+$, for some $b\in\Omega$. Consequently, $\mathfrak b(w)=\mathrm{tr}\left[bw^{-1}\right]$, $w\in\Omega_+$.

\bigskip
\subsection{Identification of $\mathfrak l$.}\label{ll}$\,$

\bigskip
Applying $D_x$ to both sides of \eqref{lambda} and referring to \eqref{dux} and \eqref{dvx} we get
\begin{equation}\label{l'}
(\alpha-\beta)zy\left(\mathfrak l'\left(u^{-1}\right)+\tfrac1\beta\mathfrak l'\left(v^{-1}\right)\right)yz=\tfrac\alpha\beta \mathfrak l'\left(v^{-1}\right)- \mathfrak l'(x)
\end{equation}
which holds for all $x,y\in\Omega_+$. 

\bigskip
Then take the $D_y$ of the above evaluating it at $h\in\Omega$. Thus, referring this time to \eqref{duy}, \eqref{dvy} and \eqref{Dy}, similarly as in the case of $\mathfrak a$, we get
\begin{align}
&(\alpha-\beta)\beta zhzx\left(\mathfrak l'\left(u^{-1}\right)+\tfrac1\beta\mathfrak l'\left(v^{-1}\right)\right)yz+(\alpha-\beta)\beta zy\left(\mathfrak l'\left(u^{-1}\right)+\tfrac1\beta\mathfrak l'\left(v^{-1}\right)\right)xzhz\label{el'}\\
&+(\alpha-\beta)zy\left(\mathfrak l''\left(u^{-1}\right)\left[h+\beta(\alpha-\beta)xzhzx\right]-\tfrac1\beta\mathfrak l''\left(v^{-1}\right)\left[(\alpha-\beta)xzhzx\right]\right)yz\nonumber\\
=&-\tfrac\alpha\beta\,\mathfrak l''\left(v^{-1}\right)\left[(\alpha-\beta)xzhzx\right].\nonumber
\end{align}

\bigskip
Note that \eqref{LIN3} implies that $\mathfrak l'(\rho w)=\mathfrak l'(w)/\rho$ and $\mathfrak l''(\rho w)=\mathfrak l'(w)/\rho^2$. Since  $y=x$ implies $u=v$ and $u^{-1}=v^{-1}=\tfrac{\alpha+1}{\beta+1}x$ it follows that 
$$
\mathfrak l'\left(u^{-1}\right)=\mathfrak l'\left(v^{-1}\right)=\tfrac{\beta+1}{\alpha+1}\,\mathfrak l'(x)\qquad\mbox{and}\qquad\mathfrak l''\left(u^{-1}\right)=\mathfrak l''\left(v^{-1}\right)=\left(\tfrac{\beta+1}{\alpha+1}\right)^2\,\mathfrak l''(x).
$$ 
Therefore \eqref{el'} assumes the form 
$$
c_1\left(x^{-1}h\mathfrak l'(x)+\mathfrak l'(x)hx^{-1}\right)+c_2\mathfrak l''(x)[h]=0,
$$
where
$$
c_1=\tfrac{\beta}{(\beta+1)^3}\,\left(1+\tfrac1\beta\right)\,\tfrac{\beta+1}{\alpha+1}=\tfrac1{(\alpha+1)(\beta+1)}\quad\mbox{and}\quad c_2=\left(\tfrac{1+\alpha/\beta}{(\beta+1)^2}+\tfrac{(\alpha-\beta)(\beta-1/\beta)}{(\beta+1)^4}\right)\,\left(\tfrac{\beta+1}{\alpha+1}\right)^2=\tfrac2{(\alpha+1)(\beta+1)}.
$$

\bigskip
Consequently 
\begin{equation}\label{l''}
	-2\mathfrak l''(x)[h]=\mathfrak l'(x)hx^{-1}+x^{-1}h\mathfrak l'(x),\quad x\in\Omega_+,\;h\in\Omega.
\end{equation}	

\bigskip
On the other hand, in view of \eqref{r2} and \eqref{LIN3} with $\rho y$ inserted for $y$, we get
\begin{equation}\label{llll}
\mathfrak l(y+\Delta_1(\rho))-\mathfrak l(y)=\mathfrak l(x+\Delta_2(\rho))-\mathfrak l(x),
\end{equation}
where for $w_1=y$ and $w_2=x$ we have
$$
\Delta_i(\rho):=(\alpha-\beta)w_i(\beta x+\rho y)^{-1}x\stackrel{\rho\to\infty}{\longrightarrow}0\,\quad i=1,2.
$$
So as $\rho\to\infty$
$$
\mathfrak l(w_i+\Delta_i(\rho))-\mathfrak l(w_i)=\mathrm{tr}\left[\mathfrak l'(w_i)\,\Delta_i(\rho)\right]+o\left(\tfrac1\rho\right),\quad i=1,2.
$$
Multiply both sides of \eqref{llll} by $\rho$ and use the above asymptotics for $\rho\to\infty$. Since 
$$
\lim_{\rho\to\infty}\,\Delta_1(\rho)=(\alpha-\beta)x\qquad\mbox{and}\qquad\lim_{\rho\to\infty}\,\Delta_2(\rho)=(\alpha-\beta)xy^{-1}x
$$ 
it follows that
\begin{equation}\label{trr}
	\mathrm{tr}\left[\mathfrak l'(y)x\right]=\mathrm{tr}\left[\mathfrak l'(x)xy^{-1}x\right],\quad x,y\in\Omega_+.
\end{equation}

\bigskip
Now we apply $D_y$ to \eqref{trr} and evaluate it at $h\in\Omega$. Thus we get
$$
\mathrm{tr}\left[\left(\mathfrak l''(y)[h]\right)x\right]=-\mathrm{tr}\left[\mathfrak l'(x)xy^{-1}hy^{-1}x\right],\quad x,y\in\Omega_+.	 
$$
Comparison with \eqref{l''} (with $x$ changed into $y$) yields
$$
\mathrm{tr}\left[(\mathfrak l'(y)hy^{-1}+y^{-1}h\mathfrak l'(y))x\right]=2\mathrm{tr}\left[\mathfrak l'(x)xy^{-1}hy^{-1}x\right],\quad h\in\Omega.
$$

\bigskip
Therefore
$$
y^{-1}x\mathfrak l'(y)+ \mathfrak l'(y)xy^{-1}=2y^{-1}x\mathfrak l'(x)xy^{-1},
$$
whence
\begin{equation}\label{yx}
x\mathfrak l'(y)y+y\mathfrak l'(y)x=2x\mathfrak l'(x)x,\quad x,y\in\Omega_+.
\end{equation}
Inserting $y=I$ in \eqref{yx} and denoting $c=\mathfrak l'(I)$ we get 
\begin{equation}\label{ll'}
\mathfrak l'(x)=\tfrac12\left(c x^{-1}+x^{-1}c\right),\quad x\in\Omega_+.
\end{equation}

\bigskip
To identify $c$ we insert $x=I$ into \eqref{yx} together with applying representation \eqref{ll'}  for  $\mathfrak l'(y)$. After cancelations, we get
$$
y^{-1}cy+ycy^{-1}=2c, \quad y\in\Omega_+,
$$
which is equivalent to $2ycy=cy^2+y^2c$. The latter yields 
\begin{equation}\label{0}
\mathrm{tr}[(cy^2+y^2c-2ycy)c]=0,\quad y\in\Omega_+.
\end{equation}
Since the left hand side of \eqref{0} can be written as 
%\begin{align*}
%&\mathrm{tr}(cy^2c)-\mathrm{tr}(ycyc)+\mathrm{tr}(cy^2c)-\mathrm{tr}(cycy)=\mathrm{tr}((cy-yc)yc)+\mathrm{tr}(cy(yc-cy))
%=\mathrm{tr}((cy-yc)yc)-\mathrm{tr}(cy(cy-yc))\\
%=&\mathrm{tr}((cy-yc)yc)-\mathrm{tr}((cy-yc)cy)=\mathrm{tr}(cy-yc)(yc-cy))=\mathrm{tr}((cy-yc)(cy-yc)^T).
%\end{align*}
$$
\mathrm{tr}[(cy-yc)(yc-cy)]=\mathrm{tr}\left[(cy-yc)(cy-yc)^T\right]
$$
it follows that $cy=yc$, $y\in\Omega_+$, which implies that  $c$ is a multiple of identity.

\bigskip
Consequently, \eqref{ll'} yields $\mathfrak l'(x)=\kappa x^{-1}$, $x\in\Omega_+$ for some $\kappa\in\mathbb R$, whence $\mathfrak l(x)=\tilde\lambda+\kappa \log\det(x)$, for some $\tilde\lambda\in\mathbb R$, $x\in\Omega_+$. In view of \eqref{LIN3} we see that $\kappa=0$. 
 
\bigskip
\subsection{Identification of $\mathfrak C_W$, $W\in\{X,Y,U,V\}$.}$\,$ 

\bigskip
Since $\mathfrak l$ is constant it follows that $\mathfrak C_W$, $W\in\{X,Y,U,V\}$, satisfy \eqref{cs} and \eqref{LIN4}-\eqref{LIN7} with  $\mathfrak l(w)=\tilde\lambda$, $w\in\Omega_+$.  Define functions $\mathfrak F_W$, $W\in\{X,Y,U,V\}$, by 
$$\mathfrak F_W(w)=\begin{cases}\mathfrak C_W(w)-\tilde\lambda\,\log\,\det(w), & W=U,Y,\\
\mathfrak C_W(w)+\tilde\lambda\,\log\,\det(w),& W=V,X,	
\end{cases} \quad  w\in\Omega_+.
$$
Then \eqref{LIN4}-\eqref{LIN7} imply $\mathfrak F_W(\rho w)=\mathfrak F_W(w)$, $w\in\Omega_+$, for $W\in\{X,Y,U,V\}$. Moreover, in view of \eqref{cs} we have
\begin{equation}\label{FFFF}
	\mathfrak F_U\left(u^{-1}\right)+\mathfrak F_V\left(v^{-1}\right)=\mathfrak F_X(x)+\mathfrak F_Y(y),\quad x,y\in\Omega_+.
\end{equation}

\bigskip
Inserting $\rho y$ instead of $y$ in the above and taking $\rho\to 0^+$, in view of continuity, we thus get
$$
\mathfrak F_U(y)+\mathfrak F_V(x)=\mathfrak F_X(x)+\mathfrak F_Y(y),\quad x,y\in \Omega_+.
$$ 
The principle of separation of variables implies $\mathfrak F_U=\mathfrak F_Y+\gamma$  and $\mathfrak F_V=\mathfrak F_X-\gamma$,
 where $\gamma$ is a real constant. 

\bigskip
Differentiating \eqref{FFFF} with respect to $x$, referring again to \eqref{dux} and \eqref{dvx}, we obtain
\begin{equation}\label{F'}
	(\alpha-\beta)\,zy\,\mathfrak F_U'\left(u^{-1}\right)yz+\tfrac\alpha\beta\, \mathfrak F_V'\left(v^{-1}\right)-\tfrac{\alpha-\beta}{\beta}\,zy\,\mathfrak F_V'\left(v^{-1}\right)yz=\mathfrak F_X'(x).
\end{equation}
Since $\mathfrak F'_W(\rho w)=\tfrac1\rho \mathfrak F'_W(w)$, $W\in\{X,Y,U,V\}$, inserting $y=\rho x$ in \eqref{F'} and using the fact that $\mathfrak F_X'=\mathfrak F_V'$ we get
$$
(\alpha-\beta)\tfrac{\rho^2}{(\beta+\rho)^2}\,\tfrac{\beta+\rho}{\rho(\alpha+\rho)}\mathfrak F_U'(x)+\tfrac{\alpha}{\beta}\,\tfrac{\beta+\rho}{\alpha+\rho}\mathfrak F_V'(x)-\tfrac{\alpha-\beta}{\beta}\,\tfrac{\rho^2}{(\beta+\rho)^2}\,\tfrac{\beta+\rho}{\alpha+\rho}\mathfrak F_V'(x)=\mathfrak F_V'(x),
$$
which, after cancellations, implies $\mathfrak F_U'=-\mathfrak F_V'$. 

\bigskip
That is, \eqref{F'} becomes the  equation \eqref{l'} with $\mathfrak l'$ changed into $\mathfrak F_U'$. Consequently, as in Section \ref{ll}, see \eqref{l''}, we get
\begin{equation}\label{F''}
	-2\mathfrak F_U''(x)[h]=\mathfrak F_U'(x)hx^{-1}+x^{-1}h\mathfrak F_U(x),\quad x\in\Omega_+,\;h\in\Omega.
\end{equation}	

\bigskip
Referring to \eqref{FFFF}  we see that
$$
\mathfrak F_U\left(u^{-1}\right)-\mathfrak F_U(y)=-\left(\mathfrak F_V\left(v^{-1}\right)-\mathfrak F_V(x)\right),\quad x,y\in\Omega_+.
$$
Inserting in this equation $\rho y$ instead of $y$ and passing to the limit when $\rho\to\infty$, analogously as in Section \ref{ll}, see \eqref{trr}, we get
$$
	\mathrm{tr}\left[\mathfrak F_U'(y)x\right]=-\mathrm{tr}\left[\mathfrak F_V'(x)xy^{-1}x\right],\quad x,y\in\Omega_+.
$$
Since $\mathfrak F_V'=-\mathfrak F_U'$ we conclude that
\begin{equation}\label{trF}
	\mathrm{tr}\left[\mathfrak F_U'(y)x\right]=\mathrm{tr}\left[\mathfrak F_U'(x)xy^{-1}x\right],\quad x,y\in\Omega_+.
\end{equation}	

\bigskip
 Therefore, following the reasoning we used in Section \ref{ll} - note that both functions $\mathfrak l$ and $\mathfrak F_U$ are invariant under scaling the argument and compare  the pair of equations \eqref{l''} and \eqref{trr} for $\mathfrak l'$ of Section \ref{ll} with  the pair \eqref{F''} and \eqref{trF} - we conclude that that $\mathfrak F_U'\equiv 0$. Consequently, $\mathfrak F_V'\equiv 0$, too. Therefore, $\mathfrak F_U=\gamma_U$ and $\mathfrak F_V=\gamma_V$, where $\gamma_W$, $W\in\{U,V\}$ are real constants, and thus $\mathfrak F_X=\gamma_V+\gamma=:\gamma_X$, $\mathfrak F_Y=\gamma_U-\gamma=:\gamma_Y$.
 
 \bigskip
 Finally, we conclude that 
 $$
 \mathfrak C_W(x)=\tilde\lambda\log\det x+\gamma_W,\quad W\in\{U,Y\} ,\quad \mbox{and}\quad \mathfrak C_W(x)=-\tilde\lambda\log\det x+\gamma_W,\quad W\in\{V,X\} ,\qquad x\in\Omega_+.
 $$

\bigskip
\subsection{Identification of distributions of $U$, $V$, $X$ and $Y$.}$\,$

\bigskip
	Let $\lambda=-\tilde\lambda +\tfrac{r+1}{2}$. Recall that 
	$$\mathfrak f_W(w)=\exp(\mathfrak g_W(w))\quad\mbox{for }\;W\in\{U,X\}\quad\mbox{and }\quad \mathfrak f_W(w)=(\det w)^{-r-1}\,\exp(\mathfrak g_W(w))\quad\mbox{for }\;W\in\{V,Y\}.$$ Referring to \eqref{gu} and $\eqref{gv}$ and the expressions derived for $\mathfrak a$, $\mathfrak b$, $\mathfrak l$ and $\mathfrak C_W$, $W\in\{U,V\} $ we get
	$$
	\mathfrak f_U(x)=e^{\gamma_1}(\det x)^{\lambda-\frac{r+1}{2}}\,e^{-\alpha\mathrm{tr}[bx]-\mathrm{tr}\left[ax^{-1}\right]}\,\mathbf 1_{\Omega_+}(x)\quad\mbox{and}\quad \mathfrak f_V(x)=e^{\gamma_2}(\det x)^{-\lambda-\frac{r+1}2}\,e^{-\mathrm{tr}[bx]-\beta \mathrm{tr}\left[ax^{-1}\right]}\,\mathbf 1_{\Omega_+}(x).
	$$ 
	Since both $\mathfrak  f_U$ and $\mathfrak f_V$ are integrable it means that  $a$ and $b$ are both in $\Omega_+$ irrespectively of the sign of  $\lambda\in\mathbb R$ and $e^{\gamma_1}$, $e^{\gamma_2}$ are the normalizing constants. Consequently, $U$ and $V$ have matrix variate GIG distributions as given in  \eqref{UV}. 
	
	\bigskip
	Similarly, \eqref{gx} and \eqref{gy} give 
	$$
	\mathfrak f_X(x)=e^{\gamma_3}(\det x)^{\lambda-\frac{r+1}2}\,e^{-\alpha\mathrm{tr}[ax]-\mathrm{tr}\left[bx^{-1}\right]}\,\mathbf 1_{\Omega_+}(x)\quad\mbox{and}\quad \mathfrak f_Y(x)=e^{\gamma_4}(\det x)^{-\lambda-\frac{r+1}2}\,e^{-\mathrm{tr}[ax]-\beta \mathrm{tr}\left[bx^{-1}\right]}\,\mathbf 1_{\Omega_+}(x),
	$$
	that is $X$ and $Y$ 
	have matrix variate GIG distributions as given in  \eqref{XY}.
\end{proof}

\bigskip
\section{{\bf BIBLIOGRAPHICAL COMMENTS AND OPEN PROBLEMS}}

\bigskip
In the original Matsumoto-Yor paper, \cite{MatYor2001}, the direct result  \eqref{oMY} was proved for $\gamma_1=\gamma_2$. The case of $\gamma_1\neq \gamma_2$ was identified in \cite{LetWes2000} and  \cite{MatYor2003}.  Interestingly, \cite{MatYor2003} also connected this property, under additional assumption that $\lambda=\tfrac12$, to hitting times of the Brownian motion with drift. This independence property is now known as the Matsumoto-Yor (MY) property, see e.g \cite{Sti2005}. It has been intensively studied in the literature. In particular, a characterization of the $\mathrm{GIG}$ and $\mathrm{Gamma}$ by the MY independence property was obtained in \cite{LetWes2000}. For its "regression" versions, where the independence of $U$ and $V$ are replaced by constancy of some conditional moments of $V$ given $U$, see \cite{ChoWan2004} and references therein.  

\bigskip
A  tree structured MY property for random vectors in $(0,\infty)^n$ (and a related characterization) was given in \cite{MasWes2004}. It appears that such tree governed independence property is, as the original one, hidden in the conditional structure of the exponential Brownian motion, see \cite{MatWesWit2009}.  A tree version of the MY property for random matrices was established in \cite{Bob2015}.

\bigskip
In \cite{MasWes2006} a modification of the MY property for positive definite matrices of different dimensions was established and, rather unexpectedly, connected with the structure of Wishart matrices: conditional independence of the block diagonal element and its Schur complement given the off-diagonal element (see also \cite{But1998}, \cite{SesWes2008}). This result, in particular, allowed to improve the celebrated characterization of the Wishart distribution  of Geiger and Heckermann in \cite{GeiHec2002} studied in the context of statistical Bayesian graphical models.

\bigskip
More recently, free probability versions of MY property and related  characterizations of free GIG (see \cite{Fer2006}) and free Poisson (which, in free probability setting, plays the role of the gamma law) were derived in \cite{Szp2017} and its regression version characterization was given in \cite{Swi2022}. 

\bigskip
It is interesting to note that very recently an ultra-discrete version of  $H_{III,B}^{(\alpha,\beta)}$ was identified in \cite{KonNakSas2025} (as well as for other functions of the $H$ class). It was proved there that this new map, as the original one, preserves independence and has the parametric YB property.

\bigskip
Finally let us discuss some related open problems:

\bigskip
\begin{enumerate}
	\item It would be interesting to know how to extend from the cone of positive definite matrices to any (Euclidean, simple) symmetric cone: the map $\phi^{(\alpha,\beta)}$,  the independence property and the characterization. A symmetric cone version of the GIG distribution has been already identifed in  \cite{Kol2017}. Since, see \eqref{faj},  
	\begin{align*}
	y(I+\alpha xy)^{-1}(I+\beta xy)&=y^{1/2}(I+\beta y^{1/2}xy^{1/2})^{1/2}(I+\alpha y^{1/2}xy^{1/2})^{-1}(I+\beta y^{1/2}xy^{1/2})^{1/2}y^{1/2},\\
	x(I+\beta yx)^{-1}(1+\alpha yx)&=x^{1/2}(I+\alpha x^{1/2}yx^{1/2})^{1/2}(I+\beta x^{1/2}yx^{1/2})^{-1}(I+\alpha x^{1/2}yx^{1/2})^{1/2}x^{1/2}),
	\end{align*}
	a plausible candidate for such an extension would be $\phi^{(\alpha,\beta)}(x,y)=(u^{(\alpha,\beta)}(x,y),\,v^{(\alpha,\beta)}(x,y))$ with
	\begin{equation}\label{uvcone}
		\begin{cases}
	u^{(\alpha,\beta)}(x,y)=&\left(P\left\{y^{1/2}\right\}\circ P\left\{\left(1+\beta P\left\{y^{1/2}\right\}x\right)^{1/2}\right\}\right)\left(1+\alpha P\left\{y^{1/2}\right\}x\right)^{-1},\\
	v^{(\alpha,\beta)}(x,y)=&\left(P\left\{x^{1/2}\right\}\circ P\left\{\left(1+\alpha P\left\{x^{1/2}\right\}y\right)^{1/2}\right\}\right)\left(1+\beta P\left\{x^{1/2}\right\}y\right)^{-1},
	\end{cases}
	\end{equation}
	where $P$ is so called quadratic representation in the symmetric cone, see e.g. \cite{FarKor1994} (in particular, $P(x)y=xyx$ in case of the cone of spd matrices).
	Such an extension was done  in \cite{Kol2017} for the original matrix variate MY property, that is in the case of  $(\alpha,\beta)=(1,0)$. Note that \eqref{uvcone} implies that $u^{(1,0)}(x,y^{-1})=(x+y)^{-1}$ and $v^{(1,0)}(x,y^{-1})^{-1}=x^{-1}-(x+y)^{-1}$. Also it would be interesting to know if the elementary, purely algebraic  proof of the YB property, given in the Appendix,  extends to function $\phi^{(\alpha,\beta)}$ defined on non-commutative algebraic domains.
	
	\bigskip
	\item Characterizations related to MY property in the matrix variate case  typically require extra hypothesis of existence and smoothness of densities, while in the univariate case no such requirements are needed. We know only one result in the literature, \cite{CasLet1996}, which falls out of such setting. The paper gives a matrix variate version of the celebrated Lukacs result, \cite{Luk1955}: independence of $(0,\infty)$ valued random variables $X$, $Y$ and independence of $X+Y$ and $\tfrac{X}{X+Y}$ characterizes gamma distributions.  However the authors required an additional assumption of invariance of distribution of the matrix analogue of the quotient $\tfrac X{X+Y}$.  Removing or weakening of  hypothesis about existence of density or, at least, of its smoothness is desirable. In this context let us mention that the original characterization in the matrix variate case from \cite{LetWes2000}  was first improved in \cite{Wes2002}, where the assumption from \cite{LetWes2000}, that the densities are $C_2$ functions, was weakened to differentiability. In \cite{Kol2017} it was  further weakened to continuity. Nevertheless, even in the case of the original  MY property, the characterization in the matrix variate case requires existence of densities which are strictly positive on $\Omega_+$.
	
	\bigskip
	\item  Quadrirational Yang-Baxter maps were introduced in \cite{AdlBobSur2004}, see also \cite{PSTV2010}. In recent years there has been a considerable  interest in construction of non-abelian quadrirational Yang-Baxter maps, see \cite{KasKouNie2025} and references therein. Since  $\phi^{(\alpha,\beta)}$, as defined in \eqref{faj}, is a YB map, it is  a natural lifting of the YB map $H^{(\alpha,\beta)}_{III,B}$ defined on $(0,\infty)^2$ to $\Omega_+^2$. It would be interesting to know if such liftings are possible in case of the remaining independence preserving quadrirational YB maps $H_I^{(\alpha,\beta)}$, $H_{II}^{(\alpha,\beta)}$ and $H_{III,A}^{(\alpha,\beta)}$. Actually, the matrix version of  $H_{II}^{(1,0)}$ was proposed in \cite{KolPil2020}, where it was proved that it is an independence preserving map for the Kummer and Wishart matrices. Similarly, a matrix version of $H_I^{(1,0)}$, which is an independence preserving map for beta distributions, was proposed in \cite{Kol2016}. 
\end{enumerate}

\bigskip
\bibliographystyle{amsplain}
\bibliography{MGIG222}

\bigskip
\appendix

\section{{\bf Proof of Theorem \ref{YBTH}}}

\begin{proof}
	Take arbitrary $x,y,z\in\Omega_+$ and define
	\begin{equation}\label{oh}\small
	\begin{cases}
	(x_1,y_1,z_1)=F_{23}^{(\beta,\gamma)}(x,y,z)=\left(x,z(I+\beta yz)^{-1}(I+\gamma yz),\,y(I+\gamma zy)^{-1}(I+\beta zy)\right),\\ \\
	(x_2,y_2,z_2)=F_{13}^{(\alpha,\gamma)}(x_1,y_1,z_1)=\left(z_1(I+\alpha x_1z_1)^{-1}(I+\gamma x_1z_1),\,y_1,\,x_1(I+\gamma z_1x_1)^{-1}(I+\alpha z_1x_1)\right),\\ \\
	(x_3,y_3,z_3)=F_{12}^{(\alpha,\beta)}(x_2,y_2,z_2)=\left(y_2(I+\alpha x_2y_2)^{-1}(I+\beta x_2y_2),\,x_2(I+\beta y_2x_2)^{-1}(I+\alpha y_2x_2),z_2\right)\\ \\
	(X_1,Y_1,Z_1)=F_{12}^{(\alpha,\beta)}(x,y,z)=\left(y(I+\alpha xy)^{-1}(I+\beta xy),\,x(I+\beta yx)^{-1}(I+\alpha yx),z\right)\\ \\
	(X_2,Y_2,Z_2)=F_{13}^{(\alpha,\gamma)}(X_1,Y_1,Z_1)=\left(Z_1(I+\alpha X_1Z_1)^{-1}(I+\gamma X_1Z_1),\,Y_1,\,X_1(I+\gamma Z_1X_1)^{-1}(I+\alpha Z_1X_1)\right),\\ \\
	(X_3,Y_3,Z_3)=F_{23}^{(\beta,\gamma)}(X_2,Y_2,Z_2)=\left(X_2,Z_2(I+\beta Y_2Z_2)^{-1}(I+\gamma Y_2Z_2),\,Y_2(I+\gamma Z_2Y_2)^{-1}(I+\beta Z_2Y_2)\right).
	\end{cases}
	\end{equation}
	
	\bigskip
	We need to  prove that 
	$$
	(x_3,y_3,z_3)=(X_3,Y_3,Z_3).
	$$
	
	In computations below we will often use the following elementary identities
	\begin{equation}\label{eleq}
	s(I+ ts)=(I+ st)s\quad \mbox{and}\quad s(I+ ts)^{-1}=(I+ st)^{-1}s,\quad s,t\in\Omega_+.
	\end{equation}
	Note that $I$ in \eqref{eleq} can be replaced by any matrix which commutes with $s$ (see the last equality in \eqref{x3}, where $A^*$ and $B^*$ commute).
	
	\subsection*{Proof of $x_3=X_3$} $\,$
	
	\bigskip
	Referring to \eqref{oh}, we see that $x_3=X_3$ is equivalent to
	$$
	z(I+\alpha X_1z)^{-1}(I+\gamma X_1z)=X_2=y_2(I+\alpha x_2y_2)^{-1}(I+\beta x_2y_2)=y_1(I+\alpha x_2y_1)^{-1}(I+\beta x_2y_1)
	$$
	
	\bigskip
	We first compute $x_2$, see \eqref{oh},
	$$
	x_2=z_1(1+\alpha xz_1)^{-1}(I+\gamma xz_1).
	$$
	We note that
	\begin{align*}
	(I+\alpha xz_1)^{-1}(I+\gamma xz_1)=&(I+\alpha xy(1+\gamma zy)^{-1}(I+\beta zy))^{-1} (I+\gamma xy(1+\gamma zy)^{-1}(I+\beta zy))\\
	=&(I+\gamma zy)[I+\gamma zy+\alpha xy(I+\beta zy)]^{-1}\,[I+\gamma zy+\gamma xy(I+\beta zy)](I+\gamma zy)^{-1}.
	\end{align*}
	Consequently,
	$$
	x_2=y(I+\beta zy)[I+\gamma zy+\alpha xy(1+\beta zy)]^{-1}\,[I+\gamma zy+\gamma xy(I+\beta zy)](I+\gamma zy)^{-1}
	$$
	
	\bigskip
	Since $$y_1=z(I+\gamma yz)(I+\beta yz)^{-1}\stackrel{\eqref{eleq}}{=}(I+\gamma zy)(I+\beta zy)^{-1}z$$
	we see that
	\begin{align*}
	x_2y_1=&y(I+\beta zy)[I+\gamma zy+\alpha xy(I+\beta zy)]^{-1}\,[I+\gamma zy+\gamma xy(I+\beta zy)](I+\beta zy)^{-1}z\\
	=&y[I+\gamma zy+\alpha(I+\beta zy) xy]^{-1}\,[I+\gamma zy+\gamma(I+\beta zy)xy]z\\
	=&y(C+\alpha Bxy)^{-1}(C+\gamma Bxy)z\\
	\stackrel{\eqref{eleq}}{=}&(C^*+\alpha B^*yx)^{-1}(C^*+\gamma B^*yx)yz,
	\end{align*}
	where we denoted $C=I+\gamma zy$ and $B=I+\beta zy$ and $^*$ stands for taking a transpose (for future reference we also denote $A=I+ \alpha zy$).
	
	\bigskip
	Since $y_2=B^{-1}Cz$, referring to $x_3$, as defined in \eqref{oh}, we can write
	\begin{align}
		x_3=&B^{-1}Cz[I+\alpha (C^*+\alpha B^*yx)^{-1}(C^*+\gamma B^*yx)yz]^{-1}[I+\beta(C^*+\alpha B^*yx)^{-1}(C^*+\gamma B^*yx)yz]\nonumber\\
		=&B^{-1}Cz[C^*+\alpha B^*yx+\alpha C^*yz+\alpha\gamma B^*yxyz]^{-1}[C^*+\alpha B^*yx+\beta C^*yz+\beta\gamma B^*yxyz]\nonumber\\
		=&z(B^*)^{-1}C^*[(A^*+\alpha  B^*yx)C^*]^{-1}B^*[C^*+yx(\alpha+\beta\gamma yz)]\nonumber\\
		=&z(B^*)^{-1}[A^*+\alpha  B^*yx]^{-1}B^*[C^*+yx(\alpha+\beta\gamma yz)]\nonumber\\
		\stackrel{\eqref{eleq}}{=}&z[A^*+\alpha yx B^*]^{-1}[C^*+yx(\alpha+\beta\gamma yz)] \label{x3}
	\end{align}
	
	\bigskip
	Now we compute $X_3=X_2$ again referring to \eqref{oh}. That is
	\begin{align*}
	X_3=&z(I+\alpha X_1z)^{-1}(I+\gamma X_1z)\\
	=&z[I+\alpha y(I+\alpha xy)^{-1}(I+\beta xy)z]^{-1}\,[I+\gamma y(I+\alpha xy)^{-1}(I+\beta xy)z]\\
	\stackrel{\eqref{eleq}}{=}&z[I+\alpha (I+\alpha yx)^{-1}(I+\beta yx)yz]^{-1}\,[I+\gamma (I+\alpha yx)^{-1}(I+\beta yx)yz]\\
	\stackrel{\eqref{eleq}}{=}&z[I+\alpha yx+\alpha (I+\beta yx)yz]^{-1}\,[I+\alpha yx+\gamma (I+\beta yx)yz]\\
	=&z[A^*+\alpha yx B^*]^{-1}\,[C^*+yx(\alpha+\beta\gamma yz)].
	\end{align*}
	which, see \eqref{x3}, agrees with $x_3$.
	
	\subsection*{Proof of $z_3=Z_3$} $\,$
	
	\bigskip
	It follows by the argument analogous to one we used above in the case $x_3=X_3$ just by changing the roles of $x$ and $z$. 
	
	\subsection*{Proof of $y_3=Y_3$} $\,$
	
	\bigskip	
	Note that $y_2=y_1$ and $Y_2=Y_1$ . 
	Since $z_3=z_2$ and $X_3=X_2$, see \eqref{oh},  we have
	\begin{equation}\label{Eqxz}
	X_3=X_2=y_1(I+\alpha x_2y_1)^{-1}(I+\beta x_2y_1)
	\quad\mbox{and}\quad 	z_3=z_2=Y_1(I+\gamma Z_2Y_1)^{-1}(I+\beta Z_2Y_1).
	\end{equation}
	
The equality  $y_3=Y_3$ is equivalent to $y_3^*=Y_3^*$ which assumes the form
	\begin{equation}\label{eq1}
	(I+\alpha x_2y_1)(I+\beta x_2 y_1x_2)^{-1}x_2=(I+\gamma Z_2Y_1)(I+\beta Z_2Y_1)^{-1}Z_2.
	\end{equation}
	Referring to \eqref{Eqxz} we see that \eqref{eq1} is equivalent to
	$$
	X_2^{-1}y_1x_2=z_2^{-1}Y_1Z_2.
	$$
	
	\bigskip
	Using \eqref{oh} for $x_2,z_2,X_2,Z_2$ we rewrite the above as
	\begin{eqnarray}\label{eq2}
	(I+\gamma X_1 z)^{-1}(I+\alpha X_1z)z^{-1}y_1z_1(I+\alpha xz_1)^{-1}(I+\gamma xz_1)\\ =(I+\alpha z_1x)^{-1}(1+\gamma z_1x)x^{-1}Y_1X_1(I+\gamma zX_1)^{-1}(I+\alpha zX_1).\nonumber
	\end{eqnarray}
	
	\bigskip
	Referring to \eqref{oh} for $y_1,z_1,Y_1,X_1$ we get
	$$
	z^{-1}y_1z_1=y=x^{-1}Y_1X_1.
	$$
	Therefore, the left hand side of \eqref{eq2} assumes the form
	\begin{align*}
	&\left[I+\gamma y(I+\alpha xy)^{-1}(I+\beta xy)z\right]^{-1}\,\left[I+\alpha y(I+\alpha xy)^{-1}(I+\beta xy)z\right]\\
	y&\left[I+\alpha x y(I+\gamma zy)^{-1}(I+\beta zy)\right]^{-1}\left[I+\gamma xy(I+\gamma zy)^{-1}(I+\beta zy)\right]\\
	\stackrel{\eqref{eleq}}{=}&\left[I+\gamma (I+\alpha yx)^{-1}(I+\beta yx)yz\right]^{-1}\left[I+\alpha (I+\alpha yx)^{-1}(I+\beta yx)yz\right]\\
	y&\left[I+\gamma xy(I+\beta zy)(I+\gamma zy)^{-1}\right]\,\left[I+\alpha x y(I+\beta zy)(I+\gamma zy)^{-1}\right]^{-1}\\
	=&\left[I+\alpha yx+\gamma(I+\beta yx)yz\right]^{-1}\left[I+\alpha yx++\alpha(I+\beta yx)yz\right]\\
	y&\left[I+\gamma zy+\gamma xy(I+\beta zy)\right]\,\left[I+\gamma zy+\alpha xy(I+\beta zy)\right]^{-1}\\
	\stackrel{\eqref{eleq}}{=}&\left[I+\alpha yx+\gamma(I+\beta yx)yz\right]^{-1}\left[I+\alpha yx++\alpha(I+\beta yx)yz\right]\\
	&\left[I+\gamma yz+\gamma yx(I+\beta yz)\right]\,\left[I+\gamma yz+\alpha yx(I+\beta yz)\right]^{-1}\,y
\end{align*}
and the right hand side of \eqref{eq2} after analogous computation gives  
\begin{align*}
	&\left[I+\alpha y(I+\gamma zy)^{-1}(I+\beta zy)x\right]^{-1}\,\left[I+\gamma y(I+\gamma zy)^{-1}(I+\beta zy)x\right]\\
	y&\left[I+\gamma zy(I+\alpha xy)^{-1}(I+\beta xy)\right]^{-1}\,\left[I+\alpha zy(I+\alpha xy)^{-1}(I+\beta xy)\right]\\
	%=&(I+\alpha (I+\gamma yz)^{-1}(I+\beta yz)yx)^{-1}(I+\gamma (I+\gamma yz)^{-1}(I+\beta yz)yx)\\
	%y&(I+\alpha zy(I+\beta xy)(I+\alpha xy)^{-1})(I+\gamma zy(I+\beta xy)(I+\alpha xy)^{-1})^{-1}\\
	%=&(I+\gamma yz+\alpha(I+\beta yz)yx)^{-1}(I+\gamma yz+\gamma(I+\beta yz)yx)\\
	%y&(I+\alpha xy+\alpha zy(I+\beta xy))(I+\alpha xy+\gamma zy(I+\beta xy))^{-1}\\
	=&\left[I+\gamma yz+\alpha(I+\beta yz)yx\right]^{-1}\,\left[I+\gamma yz+\gamma(I+\beta yz)yx\right]\\
	&\left[I+\alpha yx+\alpha yz(I+\beta yx)\right]\,\left[I+\alpha yx+\gamma yz(I+\beta yx)\right]^{-1}\,y.
	\end{align*}
	
	\bigskip
	To finish the proof it will be convenient to denote $u=yx$ and $v=yz$. After multiplying  \eqref{eq2} by $y^{-1}$ from the right,  in view of commutation of inner two elements on both sides, \eqref{eq2} can be rewritten as
	\begin{align}\label{eq3}
	&(I+\alpha u+\gamma v+\beta\gamma uv)^{-1}(I+\gamma(u+v+\beta uv))(I+\alpha(u+v+\beta uv))(I+\alpha u+\gamma v+\alpha\beta uv)^{-1}\\
	=&(I+\alpha u+\gamma v+\alpha\beta vu)^{-1}(I+\alpha(u+v+\beta vu))(I+\gamma(u+v+\beta vu))(I+\alpha u+\gamma v+\beta\gamma vu)^{-1}.\nonumber
	\end{align}	
	Writing $s=u+v+\beta uv$ and taking side-wise inverse of \eqref{eq3}  we get
	\begin{align*}
	&(I+\alpha s+(\gamma -\alpha)v)(I+\alpha s)^{-1}(I+\gamma s)^{-1}(I+\gamma s+(\alpha-\gamma)u)\\
	=&(I+\gamma s^*+(\alpha-\gamma)u)(I+\gamma s^*)^{-1}(I+\alpha s^*)^{-1}(I+\alpha s^*+(\gamma -\alpha)v).
	\end{align*}
	Performing multiplication we get the equivalent form
	\begin{align*}
	&(\gamma-\alpha)v(I+\alpha s)^{-1}+(\alpha-\gamma)(I+\gamma s)^{-1}u-(\gamma-\alpha)^2v(I+\alpha s)^{-1}(I+\gamma s)^{-1}u\\
	&=(\alpha-\gamma)u(I+\gamma s^*)^{-1}+(\gamma-\alpha)(I+\alpha s^*)^{-1}v-(\gamma-\alpha)^2u(I+\gamma s^*)^{-1}(I+\alpha s^*)^{-1}v.
	\end{align*}
	So for $\alpha=\gamma$ the equality holds true. 
	
	\bigskip
	Now we assume that $\alpha\neq\gamma$ and, denoting $s_\kappa=(I+\kappa s)^{-1}$, $\kappa=\alpha,\beta$, we rewrite the above equality as 
		$$
	vs_\alpha-s_\gamma u-(\gamma-\alpha)vs_\alpha s_\gamma u=s_\alpha^*v-u s_\gamma^*-(\gamma-\alpha)u  s_\gamma^*s_\alpha^*v.
	$$

	Its left hand side can be written as
$$
v s_\alpha-s_\alpha u+[s_\alpha-s_\gamma]u-(\gamma-\alpha)vs_\alpha  s_\gamma u=vs_\alpha-s_\alpha u+(\gamma-\alpha)(s-v)s_\alpha s_\gamma u.
$$
	and its right hand side can be written as
$$
s_\alpha^*v-u s_\alpha^*+u[s_\alpha^*-s_\gamma^*]-(\gamma-\alpha)us_\gamma^*s_\alpha^*v=
		s_\alpha^*v-us_\alpha^*+(\gamma-\alpha)us_\gamma^*s_\alpha^*(s^*-v).
$$
	
	\bigskip
	Since $s-v=u(I+bv)$ and $s^*-v=(I+bv)u$ to finish the proof it suffices to show that
	\begin{equation}\label{eq5}
	vs_\alpha-s_\alpha u=s_\alpha^*v-us_\alpha^*\quad\mbox{and}\quad (1+bv)s_\alpha s_\gamma=s_\gamma^*s_\alpha^*(I+\beta v).
		\end{equation}																																																																																							 
		The second equation of \eqref{eq5} holds since  after taking inverses sidewise it follows directly from $(I+
			\beta v)s=s^*(I+\beta v)$ (applied twice). To see that the first equality in \eqref{eq5} holds  we rewrite it in an equivalent form
		$$
		(I+\beta v)s_\alpha-s_\alpha(I+\beta u)=s_\alpha^*(I+\beta v)-(I+\beta u)s_\alpha^*
		$$
		and refer to  $(I+\beta v)s=s^*(I+\beta v)$ and $s(I+\beta u)=(I+u)s^*$. 
		
		\bigskip
		Thus the proof of $Y_3=y_3$ is finished.
\end{proof}

\bigskip
{\bf Acknowledgement:}  We are grateful to  M. Sasada for comments on possible extensions of quadrirational YB maps to matricial domains. We thank Y. Goyotoku for the reference \cite{KasKouNie2025}. 

JW was supported by  National Science Center Poland [project no. 2023/51/B/ST1/01535].

\bigskip
$\,$
\end{document}